\documentclass{amsart}[12pt]

\newtheorem{theorem}{Theorem}

\newtheorem{corollary}[theorem]{Corollary}

\theoremstyle{definition}

\newtheorem{remark}[theorem]{Remark}

\usepackage{amscd,amssymb}

\begin{document}

\title
[Graded algebras with prescribed Hilbert series]
{Graded algebras with prescribed Hilbert series}
\author[Vesselin Drensky]{Vesselin Drensky}
\date{}
\address{Institute of Mathematics and Informatics,
Bulgarian Academy of Sciences,
Acad. G. Bonchev Str., Block 8,
1113 Sofia, Bulgaria}
\email{drensky@math.bas.bg}

\thanks
{Partially supported by Grant KP-06 N 32/1 of 07.12.2019
``Groups and Rings -- Theory and Applications'' of the Bulgarian National Science Fund.}

\maketitle

\begin{abstract}
For any power series $a(t)$ with exponentially bounded nonnegative integer coefficients
we suggest a simple construction of a finitely generated monomial associative algebra $R$ with Hilbert series
$H(R,t)$ very close to $a(t)$.
If $a(t)$ is rational/algebraic/transcendental, then the same is $H(R,t)$.
If the growth of the coefficients of $a(t)$ is polynomial, in the same way we construct a graded algebra $R$
preserving the polynomial growth of the coefficients of its Hilbert series $H(R,t)$.
Applying a classical result of Fatou from 1906 we obtain that if a finitely generated graded algebra $R$ has a finite Gelfand-Kirillov dimension, then
its Hilbert series is either rational or transcendental.
In particular the same dichotomy holds for the Hilbert series of finitely generated algebras $R$ with polynomial identity.
\end{abstract}

\section*{Introduction}
We consider finitely generated unitary associative algebras $R$ over an arbitrary field $K$ of any characteristic.
The algebra $R$ is {\it graded} if $R$ is a direct sum of vector subspaces $R_0,R_1,R_2,\ldots$ called {\it homogeneous components} of $R$
and
\[
R_mR_n\subset R_{m+n},\quad m,n=0,1,2,\ldots.
\]
In the sequel we consider graded algebras only.
We assume that $R_0=0$ or $R_0=K$ and the generators of $R$ are of first degree.
The formal power series
\[
H(R,t)=\sum_{n\geq 0}\dim(R_n)t^n,
\]
is called the {\it Hilbert series} of $R$. In the sequel, when we speak about algebraic and transcendental power series
we shall mean over ${\mathbb Q}(t)$.

Let $K\langle X_d\rangle=K\langle x_1,\ldots,x_d\rangle$ be the free $d$-generated unitary associative algebra
and let $\langle X_d\rangle$ be the set of all monomials in $K\langle X_d\rangle$.
Any $d$-generated algebra $R$ is a homomorphic image of $K\langle X_d\rangle$ modulo an ideal $I$.
If the ideal $I$ is finitely generated we say that $R$ is {\it finitely presented}. If
\[
U=\{u_j\in \langle X_d\rangle\mid j\in J\}
\]
is a set of monomials, then the factor algebra $R=K\langle X_d\rangle/I$ of $K\langle X_d\rangle$
modulo the ideal $I=(U)$ generated by $U$ is a {\it monomial algebra}.

In the case of polynomial growth of the coefficients of the Hilbert series $H(R,t)$
a precise way to measure the growth is by the Gelfand-Kirillov dimension.
If $R$ is an algebra (not necessarily graded) generated by a finite dimensional vector space $V$, then the {\it growth function}
of $R$ is defined by
\[
g_V(n)=\dim(R^n),\quad R^n=V^0+V^1+V^2+\cdots+V^n,\quad n=0,1,2,\ldots,
\]
and the {\it Gelfand-Kirillov dimension} is
\[
\text{GKdim}(R)=\limsup_{n\to\infty}\log_n(g_V(n)).
\]
For a background on Gelfand-Kirillov dimension see the book by Krause and Lenagan \cite{KL}.
For finitely generated commutative algebras the Gelfand-Kirillov dimension is always an integer.
In the noncommutative case $\text{GKdim}(R)\in \{0,1\}\cup [2,\infty)$.
By the Bergman Gap Theorem $\text{GKdim}(R)\not\in (1,2)$.
Borho and Karft \cite{BK} constructed examples of algebras $R$ such that $\text{GKdim}(R)=\alpha$
for any positive real number $\alpha\geq 2$.

Govorov \cite{G1} proved that if the set of monomials $U$ is finite, then
the Hilbert series of the monomial algebra $R=K\langle X\rangle/(U)$ can be expressed as a rational function.
He conjectured \cite{G1, G2} that the same holds for the Hilbert series of finitely presented graded algebras.
Shearer \cite{Sh} constructed a finitely presented graded algebra with algebraic nonrational Hilbert series.
As he mentioned the same construction gives also an example with a transcendental Hilbert series.
A simpler example of finitely presented algebra with algebraic nonrational Hilbert series was given by Kobayashi \cite{K}.

It is well known that if the Hilbert series $H(R,t)$ is algebraic,
then its coefficients grow either exponentially or polynomially.
Power series with {\it intermediate  growth} (faster than polynomial and slower than exponential)
are transcendental. In \cite{G1} Govorov constructed also a two-generated monomial algebra
such that the sequence of the dimensions $\dim(R_n)$
grows intermediately, i.e., its Hilbert series is not algebraic.
Other examples of finitely generated associative algebras
with Hilbert series with coefficients of intermediate growth are universal enveloping algebras of infinite dimensional Lie algebras
of polynomial growth, see Smith \cite{Sm} and Lichtman \cite{L}.
Petrogradsky [P] introduced a refined scale for measuring the growth of algebras with intermediate growth.

The algebras in the examples of Smith \cite{Sm}, Lichtman \cite{L}, and Petrogradsky \cite{P} are not finitely presented.
There was a conjecture of Borho and Kraft [BK] that finitely
presented associative algebras cannot be of intermediate growth.
For a counterexample it is sufficient to show that there exists a finitely presented and
infinite dimensional Lie algebra with polynomial growth. The
easiest example is the Witt algebra $L$ of the derivations of $K[z]$.
The first example of a finitely presented graded algebra with Hilbert series with intermediate growth of the coefficients
was given by Ufnarovskij \cite{U}. (In his example the algebra is two-generated by elements of degree 1 and 2.)
See also the recent paper by Ko\c{c}ak \cite{Ko} for more examples and a survey on finitely presented algebras of intermediate growth.

A result of Macaulay \cite{M} gives that the coefficients of the Hilbert series of a finitely generated commutative algebras
are a subject of many restrictions. It has turned out that the situation is completely different for noncommutative algebras.
In the present paper we give a construction of a monomial algebra $R$ with Hilbert series which is close to an arbitrary
power series $a(t)$ with nonnegative integer coefficients and exponentially bounded growth of the coefficients.
Our approach is in the spirit of the approach of Borho and Kraft \cite{BK}
and its modification in the book of the author \cite[Theorem 9.4.11]{D}.
Using the same ideas we prove a version for graded algebras.
The constructions transfers the properties of $a(t)$ to $H(R,t)$. If $a(t)$ is rational, algebraic or transcendental, then the same is $H(R,t)$.

A classical theorem of Fatou \cite{F} from 1906 gives that if the coefficients of the power series $a(t)$
with integer coefficients are of polynomial growth and
$a(t)$ is algebraic, then it is rational. This immediately implies that the Hilbert series of graded algebras of finite Gelfand-Kirillov dimension
are either rational or transcendental. A theorem of Berele \cite{B} states that finitely generated algebras with polynomial identity (or PI-algebras)
are of finite Gelfand-Kirillov dimension. As a consequence we obtain that the same dichotomy holds
for the Hilbert series of finitely generated graded PI-algebras.

\section{The construction}

In this section we present constructions for monomial and graded algebras with Hilbert series close to a preliminary given power series.

\begin{theorem}\label{main theorem}
Let
\[
a(t)=\sum_{n\geq 0}a_nt^n
\]
be a power series with nonnegative integer coefficients. Let $d$ be a positive integer such that $a_n\leq d^n$, $n=0,1,2,\ldots$.
Then for any integer $p=0,1,2$, there exists a $(d+1)$-generated monomial algebra $R$ such that its Hilbert series is
\begin{equation}\label{the first case of Hilbert series}
H(R,t)=\frac{1}{1-dt}+\frac{t}{(1-dt)^2}+\frac{t^2a(t)}{(1-dt)^p}.
\end{equation}
\end{theorem}

\begin{proof}
We shall work in the $(d+1)$-generated free algebra $K\langle X_d,y\rangle$.
Since the dimension of the homogeneous component $K\langle X_d\rangle_n$ of degree $n$
of the free subalgebra $K\langle X_d\rangle$ of $K\langle X_d,y\rangle$ is equal to $d^n$
we can choose a subset $A$ of monomials such that
\[
A=A_0\cup A_1\cup A_2\cup\cdots,\quad A_n\subset \langle X_d\rangle_n,\quad \vert A_n\vert=a_n,\quad n=0,1,2,\ldots.
\]
Let the subset $U$ of $\langle X_d,y\rangle$ consist of the following monomials
\begin{equation}\label{monomial defining relations}
yu_1(X_d)yu_2(X_d)y,\quad yw(X_d)y,\quad u_1(X_d),u_2(X_d)\in \langle X_d\rangle, \quad w(X_d)\in \langle X_d\rangle\setminus A.
\end{equation}
Then the monomial algebra $R=K\langle X_d,y\rangle/(U)$ has a basis consisting of
\begin{equation}\label{basis of R}
u(X_d),\quad u_1(X_d)yu_2(X_d),\quad u_1(X_d)yv(X_d)yu_2(X_d),
\end{equation}
\[
u(X_d),u_1(X_d),u_2(X_d)\in \langle X_d\rangle,\quad v(x_d)\in A.
\]
Using that the Hilbert series of $K\langle X_d\rangle$ is equal to $\displaystyle \frac{1}{1-dt}$
we obtain that the Hilbert series of $R$ is
\[
H(R,t)=\frac{1}{1-dt}+\frac{t}{(1-dt)^2}+\frac{t^2a(t)}{(1-dt)^2}.
\]
This is the Hilbert series (\ref{the first case of Hilbert series}) for $p=2$.
Adding to the relations (\ref{monomial defining relations}) the monomials $x_iyv(X_d)y$, $x_i\in X_d$, $v(X_d)\in A$,
we remove from the basis (\ref{basis of R}) the monomials $u_1(X_d)yv(X_d)yu_2(X_d)$ with $u_1(X_d)\not=1$.
Then the Hilbert series of $R$ is
(\ref{the first case of Hilbert series}) for $p=1$.
Finally, adding to the relations also the monomials $yv(X_d)yx_i$, $x_i\in X_d$, $v(X_d)\in A$, we handle also the case $p=0$.
\end{proof}

\begin{corollary}\label{corollary for graded algebras}
Let $a(t)$ be a power series which satisfies the assumptions of Theorem \ref{main theorem}.
Then for any nonnegative integers $p,q$, $p+q\leq 2$, there exists a $(d+1)$-generated graded algebra $R$ such that its Hilbert series is
\[
H(R,t)=\frac{1}{1-dt}+\frac{t}{(1-dt)^2}+\frac{t^2a(t)}{(1-dt)^p(1-t)^{dq}}.
\]
\end{corollary}

\begin{proof}
We add to (\ref{monomial defining relations}) the relations
\[
x_{i_{\sigma(1)}}\cdots x_{i_{\sigma(n)}}yv(X_d)y=x_{i_1}\cdots x_{i_n}yv(X_d)y,
\]
where $x_{i_1}\cdots x_{i_n}\in \langle X_d\rangle$, $v(X_d)\in A$, and $\sigma$ runs on the symmetric group
$S_n$ of degree $n$. Then the Hilbert series of the graded algebra $R$ becomes
\[
H(R,t)=\frac{1}{1-dt}+\frac{t}{(1-dt)^2}+\frac{t^2a(t)}{(1-dt)(1-t)^d},
\]
which is the case $p=q=1$.
With similar arguments, as in the proof of Theorem \ref{main theorem} we produce examples of graded algebras for the other cases $p+q\leq 2$.
\end{proof}

\begin{remark}\label{remark xwx}
If we add to the monomials (\ref{monomial defining relations}) the relations $x_iyx_j$, $x_i,x_j\in X_d$
(and the relations $x_iy^2$ and $y^2x_j$ if $1\in A$), then we shall construct a monomial algebra $R$ with Hilbert series
\[
H(R,t)=\frac{1+2t}{1-dt}-t+t^2a(t).
\]
\end{remark}

\begin{theorem}\label{polynomially bounded graded algebras}
Let $a(t)$ be a power series with nonnegative integer coefficients $a_n$. Let $d$ be a positive integer such that
$\displaystyle a_n\leq \binom{d+n-1}{n-1}$, $n=0,1,2,\ldots$.
Then for any integer $p=0,1,2$, there exists a $(d+1)$-generated graded algebra $R$ such that its Hilbert series is
\[
H(R,t)=\frac{1}{(1-t)^d}+\frac{t}{(1-t)^{2d}}+\frac{t^2a(t)}{(1-t)^{dp}}.
\]
\end{theorem}

\begin{proof}
We start with a graded algebra $S$ which is a factor algebra of the free algebra $K\langle X_d,y\rangle$ modulo the relations
$x_ix_j=x_jx_i$, $x_i,x_j\in X_d$. Hence the monomials in $X_d$ in $S$ behave as the set $[X_d]$ of monomials in the polynomial algebra $K[X_d]$.
As in the proof of Theorem \ref{main theorem}, working in $S$,
we choose a subset $A=A_0\cup A_1\cup A_2\cup\cdots$ such that the monomials $A_n$ are of degree $n$ and $\vert A_n\vert=a_n$.
Adding the relations
\[
yu_1(X_d)yu_2(X_d)y,\quad yw(X_d)y,\quad u_1(X_d),u_2(X_d)\in [X_d],\quad w(X_d)\in [X_d]\setminus A,
\]
we obtain a graded algebra $R$ with Hilbert series
\[
H(R,t)=\frac{1}{(1-t)^d}+\frac{t}{(1-t)^{2d}}+\frac{t^2a(t)}{(1-t)^{2d}}.
\]
The other two cases $p=0$ and $p=1$ are completed as the corresponding cases in Theorem \ref{main theorem}.
\end{proof}

\section{Dichotomy for graded algebras of polynomial growth}

The condition that a power series with nonnegative integer coefficients is algebraic is very restrictive.
We shall use the following partial case of a theorem of Fatou \cite{F} from 1906.

\begin{theorem}\label{Theorem of Fatou}
If the coefficients of a power series are integers and are bounded polynomially,
then the series is either rational or transcendental.
\end{theorem}

Since the coefficients of the Hilbert series of graded algebras of finite Gelfand-Kirillov dimension grow polynomially,
as an immediate consequence of Theorem \ref{Theorem of Fatou} we obtain:

\begin{theorem}\label{dichotomy of finite GKdim}
The Hilbert series of a finitely generated graded algebra of finite Gelfand-Kirillov dimension is either rational or transcendental.
\end{theorem}

The element $f(x_1,\ldots,x_n)\in K\langle X\rangle=K\langle x_1,x_2,\ldots\rangle$ is a {\it polynomial identity} for the algebra $R$
if $f(r_1,\ldots,r_n)=0$ for all $r_1,\ldots,r_n\in R$. If $R$ satisfies a nontrivial polynomial identity, it is called a {\it PI-algebra}.
The following theorem from 1982 is due to Berele \cite{B}.

\begin{theorem}\label{GKdim of PI-algebra}
Finitely generated PI-algebras are of finite Gelfand-Kirillov dimension.
\end{theorem}

\begin{remark}
It is well known that the class of PI-algebras has nice structure and combinatorial theory.
From many points of view finitely generated PI-algebras are similar to commutative algebras.
Theorem \ref{GKdim of PI-algebra} is a confirmation of this similarity.
Nevertheless there are many differences. For example, the Gelfand-Kirillov dimension of a finitely generated commutative algebra is an integer.
On the other hand all examples of finitely generated PI-algebras of Gelfand-Kirillov dimension $\alpha\geq 2$ from \cite{BK}
are tensor products $K[y_1,\ldots,t_m]\otimes_KR$ where $R$ is a two-generated algebra $R$ of Gelfand-Kirillov dimension in the interval $[2,3]$.
The algebra $R$ satisfies the polynomial identity
\[
(x_1x_2-x_2x_1)(x_3x_4-x_4x_3)(x_5x_6-x_6x_5)=0
\]
and the same identity is satisfied by the tensor product.
For comparison, the examples in \cite[Theorem 9.4.11]{D} are two-generated and satisfy the polynomial identity
\[
(x_1x_2-x_2x_1)\cdots(x_{2m-1}x_{2m}-x_{2m}x_{2m-1})=0
\]
for a suitable $m$.
\end{remark}

The combination of Theorems \ref{dichotomy of finite GKdim} and \ref{GKdim of PI-algebra} gives:
\begin{theorem}\label{dichotomy Hilbert series of PI-algebras}
The Hilbert series of a finitely generated graded PI-algebra is either rational or transcendental.
\end{theorem}

\section{Concluding remarks}
As we have mentioned in the introduction, if the power series $a(t)$ is rational, algebraic or transcendental, the same property has
the Hilbert series of the monomial and graded algebras $R$ constructed in Theorems \ref{main theorem} and \ref{polynomially bounded graded algebras}.
It is a natural question where to find algebraic and transcendental power series
$a(t)$ with nonnegative integer coefficients $a_n$. It is easy to construct transcendental power series. We shall discuss three well known ways.

As in the paper by Smith \cite{Sm},
if the coefficients $b_n$ of the power series $\displaystyle b(t)=\sum_{n\geq 1}b_nt^n$ with nonnegative integer coefficients
grow subexponentially and $b(t)$ is not a polynomial, then the expansion into a power series of the infinite product
\[
a(t)=\prod_{n\geq 1}\frac{t^n}{(1-t^n)^{b_n}}
\]
is also of subexponential (and not polynomial) growth. Hence $a(t)$ is transcendental.
The most famous example is the generating function
\[
p(t)=\prod_{n\geq 1}\frac{1}{1-t^n}=\sum_{n\geq 0}p_nt^n
\]
which counts the number $p_n$ of the partitions of $n$. Its asymptotics
\[
p_n\approx \frac{1}{4n\sqrt{3}}\exp\left(\pi\sqrt{\frac{2}{3}n}\right)
\]
was found by Hardy and Ramanujan \cite{HR} in 1918 and independently by Uspensky \cite{Us} in 1920.
The power series $p(t)$ is equal to the Hilbert series of the example of Ufnarovsky \cite{U}.

The theorem of Mahler \cite[p. 42]{Ma} provides other examples of transcendental power series --
the lacunary series with nonnegative integer coefficients. Recall that the power series $a(t)$ is {\it lacunary}, if
\[
a(t)=\sum_{k\geq 1}a_{n_k}t^{n_k},\quad a_{n_k}\not=0,\quad \lim_{k\to\infty}(n_{k+1}-n_k)=\infty.
\]
Maybe the best known example of such series is
\[
a(t)=\sum_{n\geq 1}t^{n!}
\]
which gives rise to the first explicitly given transcendental number $\displaystyle a\left(\frac{1}{10}\right)$,
the constant of Liouville \cite{Li}. Another example also based on the result of Mahler is given in the book of Nishioka
\cite[Theorem 1.1.2]{N}
\[
a(t)=\sum_{n\geq 0}t^{d^n},\quad d\geq 2.
\]

Finally, we may construct transcendental series with
polynomial or exponential growth of the coefficients using multiplicative functions $\alpha:{\mathbb N}\to {\mathbb C}$.
B\'ezivin \cite{Be} described the functions $\alpha$ such that the generating function $a(t)$ of the sequence $a_n=\alpha(n)$ is algebraic.
Combining with results of Leitmann and Wolke \cite{LW}
it follows that a power series whose coefficients are multiplicative is either transcendental or rational.
S\'ark\"ozy \cite{Sa} described multiplicative functions such that $a(t)$ is rational. As a consequence it is easy to construct multiplicative functions
such that the corresponding generating function is transcendental. See Bell, Bruin, Coons \cite{BBC} for details.
A simple example of transcendental generating function $a(t)$ is if define the multiplicative function $\alpha$ on prime numbers by
$\alpha(p)=q$, where the $q$'s are pairwise different primes and $\alpha(p)\not=p$ for all prime $p$.

It is more difficult to construct algebraic power series with nonnegative integer coefficients.
We shall mention two methods only.

Recently there are applications to monomial algebras of the theory of regular languages
and the theory of finite-state automata
which give new results and new proofs of old results providing algebras with rational and algebraic nonrational Hilbert series,
see La Scala \cite{LS} and La Scala, Piontkovski, and Tiwari \cite{LSPT} and the references there.

Another possible way is to consider the generating function which counts the planar rooted trees with given number of leaves
and fixed number of incoming branches in each vertex, see, e.g. Drensky and Holtkamp \cite{DH}. The simplest example is
the generating function which counts binary planar rooted trees (enumerating also the Catalan numbers).
The forthcoming paper by Drensky and Lalov \cite{DL} generalizes the methods of \cite{DH} and gives more examples
of algebraic power series with nonnegative integer coefficients.

\end{document}